\documentclass[preprint]{imsart}

\RequirePackage{fix-cm}

\usepackage{latexsym}
\usepackage{array}
\usepackage{color}
\usepackage{rotating}
\usepackage{epsfig, wrapfig}
\usepackage{graphicx, lscape, graphics, color} 
\usepackage{verbatim}
\usepackage{latexsym}
\usepackage{fancyvrb}
\usepackage{pdfpages}
\usepackage{fmtcount}
\usepackage{multirow}
\usepackage{url}
\usepackage[english]{babel}
\usepackage[utf8]{inputenc}
\usepackage{tikz} 
\usetikzlibrary{arrows, decorations.pathmorphing, backgrounds, fit, positioning, shapes.symbols, chains}
\usetikzlibrary{decorations.markings}
\usepackage{lscape}
\usepackage{epstopdf}
\usepackage{algpseudocode}
\usepackage{algorithm}
\usepackage{kotex}
\usepackage{soul}
\usepackage{enumerate}
\usepackage{colortbl}
\usepackage{booktabs}
\usepackage{appendix}
\usepackage{mathtools}
\usepackage{tikz} 
\usetikzlibrary{arrows, decorations.pathmorphing, backgrounds, fit, positioning, shapes.symbols, chains}
\usetikzlibrary{decorations.markings}
\usetikzlibrary{arrows.meta}

\def\Frechet{Fr\'{e}chet}

\RequirePackage{amsthm, amsmath, amsfonts, amssymb}
\RequirePackage[authoryear]{natbib}
\usepackage{hyperref}

\theoremstyle{plain}

\newtheorem{theorem}{Theorem}
\newtheorem{lemma}{Lemma}
\newtheorem{definition}{Definition}
\newtheorem{remark}{Remark}

\newtheorem{example}{Example}
\newtheorem{proposition}{Proposition}
\theoremstyle{remark}

\begin{document}

\begin{frontmatter}
\title{On the Kolmogorov-Feller weak law of large numbers for the Fr\'{e}chet mean on non-compact symmetric spaces}
\runauthor{Lee and Jung}
\runtitle{On the Kolmogorov-Feller weak law of large numbers for the Fr\'{e}chet mean}

\begin{aug}
\author[A]{\fnms{Jongmin}~\snm{Lee}\ead[label=e1]{jongmin.lee@pusan.ac.kr}\orcid{0000-0003-1723-4615}}
\and
\author[B]{\fnms{Sungkyu}~\snm{Jung}\ead[label=e2]{sungkyu@snu.ac.kr}\orcid{0000-0002-6023-8956}}
\address[A]{Department of Statistics,
Pusan National University\printead[presep={,\ }]{e1}}

\address[B]{Department of Statistics and Institute for Data Innovation in Science,
Seoul National University\printead[presep={,\ }]{e2}}
\end{aug}

\begin{abstract}
We prove the Kolmogorov–Feller weak law of large numbers for sample Fr\'{e}chet means on non-compact symmetric spaces. The result covers independent, non-identically distributed data, extending beyond the i.i.d. setting. Examples of symmetric positive-definite matrices and product symmetric spaces are provided.
\end{abstract}



\end{frontmatter}
For a sequence of independent and identically distributed real random variables, $(X_i)_{i=1}^{\infty}$, the celebrated Kolmogorov--Feller weak law of large numbers 
\citep[][pp. 234-236]{feller1971introduction} 
says
 \begin{equation}\label{KF_wlln}
  n\mathbb{P}(|X_1|>n) \underset{n \to \infty}\to 0 ~ \mbox{if and only if} ~ \frac{\sum_{i=1}^{n} X_i - n \mathbb{E}[X_1 1_{|X_1|\le n}]}{n} \overset{\mathbb{P}}{\to} 0 \quad \mbox{as} ~ n \to \infty.
\end{equation}
The result in (\ref{KF_wlln}) can be viewed as an extension of the weak law of large numbers for random variables whose distributions are possibly too heavy-tailed to admit a finite first moment (i.e., $\mathbb{E}|X_1| = \infty$), yet still exhibits moderate tail decay in the sense that $n\mathbb{P}(|X_1|>n) \underset{n\to \infty}{\to} 0$, which is called the Kolmogorov--Feller tail condition. 

The Kolmogorov--Feller weak law of large numbers 
has been extended to various settings. For example, it is well-known that \eqref{KF_wlln} holds for an independent but not necessarily identically distributed sequence $(X_i)_{i=1}^{\infty}$ of real-valued random variables \citep[][Theorem 2.2.11]{durrett2019probability}. \cite{stoica2010kolmogorov} investigate the case where $(X_i)_{i=1}^{\infty}$ is assumed to be exchangeable, and a conditional version of the extended Kolmogorov--Feller weak law of large numbers is established by \cite{yuan2015conditional}. 
 Moreover, under a symmetry condition, (\ref{KF_wlln}) 
has been extended for random variables in
Banach spaces; see Corollary 3.3 in \cite{marcus1979stable} and Theorem 9.22 in \cite{ledoux1991probability}. However, analogous extensions to nonlinear spaces appear to be unavailable.

Beyond vector spaces, suppose a data set $\{x_1, x_2, \ldots x_n\}$ lies in a metric space $(M,d)$. For the central tendency of the data set, it is natural to consider the sample \Frechet\ mean, $$\mu_n:=\mbox{argmin}_{m\in M} \frac{1}{n} \sum_{i=1}^{n}d^2(x_i, m),$$ which generalizes the arithmetic mean to any metric space.
Theoretical properties of the sample \Frechet\ mean, such as the law of large numbers, central limit theorem, and non-asymptotic bounds, have been established in Riemannian manifolds and general metric spaces (see, e.g., \cite{ziezold1977expected, sturm2003probability, bhattacharya2003large, schotz2022strong, evans2024limit, brunel2024concentration, schotz2025variance, brunel2025finite}). However, their validity under heavy-tailed distributions that may not possess a finite first moment has not been examined. Moreover, only a few studies to date \citep[e.g.,][]{brunel2024concentration, brunel2025finite} have addressed the case of independent random variables with potentially distinct distributions.

To address this problem, we prove a 
Kolmogorov--Feller-type weak law of large numbers 
for sequences of random variables lying in a specific class of metric spaces. The obstacles for such derivations are two-fold: (i) the possibility of non-uniqueness of \Frechet\ means (for instance, compact manifolds such as the sphere $S^k$ may exhibit non-unique \Frechet\ means), and (ii) the difficulty in evaluating $\mathbb{E}[X_{1}1_{|X_1|\le n}]$ appearing in the right-hand side of (\ref{KF_wlln}). 
We circumvent the non-uniqueness issue by restricting the domain to be a symmetric space of non-compact type (hereafter referred to as a \textit{non-compact symmetric space}), where the sample \Frechet\ mean is known to be unique (e.g., Proposition 4.3 in \cite{sturm2003probability}), albeit the population \Frechet\ mean may not exist. 
Representative examples include Euclidean space $\mathbb{R}^k$, hyperbolic space $\mathbb{H}^{k}$, the space of symmetric positive-definite matrices $(\text{Sym}^{+}(k), ~ d_{\text{\tiny AI}})$ equipped with the affine-invariant metric, and their product spaces.

The remedy for (ii) is to impose a symmetry condition on the underlying probability distribution $\mathbb{P}_{X}$, as done for the Banach space extension of  \cite{marcus1979stable}.
The symmetry condition, defined precisely in the next section, ensures that there exists a center of symmetry $\mu \in M$ corresponding to $\mathbb{P}_{X}$, and we take $\mu$ as a reference point. To generalize (\ref{KF_wlln}) to our setting, the strategy is to replace $|X_1|$ with $d(X_1, \mu)$, and the sample mean $n^{-1}\sum_{i=1}^{n}X_i$ with the sample \Frechet\ mean $\mu_n$. 

We establish a modified Kolmogorov--Feller tail condition, which is shown to be sufficient for the convergence of the sample \Frechet\ mean in probability, under the symmetry assumption.  An extension to independent but not necessarily identically distributed random variables is discussed as well. 

\section{Kolmogorov--Feller's weak law of large numbers for Fréchet means}

A simply connected and geodesically complete Riemannian manifold is said to be a {\it Hadamard manifold} if it has non-positive curvatures everywhere. A connected, complete and smooth Riemannian manifold $(M,d)$ is said to be a {\it symmetric space} if for any $\nu \in M$, there exists a unique isometry $s_\nu$ (called geodesic symmetry) that fixes $\nu$ and reverses all geodesics through $\nu$; that is, for each $x \in M$, $s_{\nu}(x)=\mbox{Exp}_{\nu}\{-\mbox{Log}_{\nu}(x)\}$, provided it is well-defined. It is known that any non-compact symmetric space is simply connected and has non-positive curvatures everywhere \citep[][]{helgason1979differential}; thus, it is Hadamard. 
%
Moreover, for $M$ that is Hadamard, the Riemannian logarithmic map at $p$ for any $p \in M$ is a diffeomorphism between $M$ and $T_{p}M \cong \mathbb{R}^{k}$ by the Cartan-Hadamard theorem. 

The notion of geodesic symmetry of a probability distribution $\mathbb{P}_{X}$ is defined as follows.
\begin{definition}
Given a symmetric space $(M,d)$ and a probability distribution $\mathbb{P}_{X}$ on the space, if there exists a $\mu \in M$ such that the distribution of $s_\mu(X)$, $\mathbb{P}_{s_\mu (X)}$, equals $\mathbb{P}_X$, then $\mathbb{P}_{X}$ is said to be geodesically symmetric about $\mu$. 
\end{definition}
A natural question that arises from the definition is whether a given distribution $\mathbb{P}_{X}$ can have two or more points of geodesic symmetry. The answer is negative under the following non-compactness condition.









\begin{proposition}\label{prop:symmetry}
Let $(M,d)$ be a non-compact symmetric space. Suppose that $\mathbb{P}_{X}$ is geodesically symmetric about both $\mu_1, \mu_2 \in M$. Then, $\mu_1=\mu_2$.  
\end{proposition}

The proof of Proposition~\ref{prop:symmetry} is given in Appendix A. Proposition~\ref{prop:symmetry} implies that any distribution $\mathbb{P}_{X}$ on $M$ admits at most one center of symmetry $\mu$. Geodesically symmetric distributions on $M$ include isotropic distributions as well as elliptically symmetric distributions, which constitute a broad class of probability distributions.
The assumption of geodesic symmetry mirrors the symmetry assumption imposed for similar results for Banach spaces \citep{marcus1979stable}. In particular, when $M$ is a vector space, $X$ is geodesically symmetric about $\mu$ if and only if $X -\mu \stackrel{d}{=} -(X - \mu)$.

Unless mentioned otherwise, hereafter we assume that $(M, d)$ is a $k$-dimensional non-compact symmetric space. 
To describe the main results, we assume the following condition. 
\begin{enumerate}
 \item[(C)] For a fixed $\mu \in M$, $(X_{i})_{i=1}^{\infty}$ is a sequence of independent $M$-valued random variables and for each $i\ge 1$, $\mathbb{P}_{X_{i}}$ is geodesically symmetric about $\mu$.
\end{enumerate}
Condition (C) imposes only independence on the sequence $(X_i)_{i=1}^{\infty}$, allowing the random variables to have different distributions that are geodesically symmetric about $\mu$.

We denote the sample \Frechet\ mean of $X_1, X_2, \ldots, X_n$ by 
$\mu_n = \mu_{n}(X_1, X_2, \ldots, X_{n})=\mbox{argmin}_{m\in M}n^{-1}\sum_{i=1}^{n} d^2(X_i,m)$.
Since $M$ is a Hadamard manifold, $\mu_n$ is guaranteed to exist and is unique \citep[][Proposition 4.3]{afsari2011riemannian}. On the other hand, under Condition (C), the population \Frechet\ mean, which is the minimizer $p \in M$ of the \Frechet\ function $F(p):= \mathbb{E}[d^2(X_1,p)]$, may not exist, since the second moment $F(p)$ itself may be infinite everywhere. 
Nevertheless, due to the symmetry about $\mu$ imposed on $(X_i)_{i=1}^{\infty}$, one may intuitively expect that the distribution of $\mu_{n}$ concentrates around the center $\mu$ as $n$ tends to infinity, under suitable tail conditions.  
%
%
To state our conditions,  we define, for \textit{any} reference point $\mu_0 \in M$, a triangular array of random variables $\{\overline{X_{n, i}}(\mu_0)\} _{n\ge i\ge 1}$ truncated at distance $n$ from $\mu_0$: 
$$\overline{X_{n, i}}(\mu_0) := \begin{cases} X_i \quad \mbox{if} ~ d(X_{i}, \mu_0) \le n, \\ \mu_0 \qquad \mbox{otherwise}, \end{cases}$$ 
for any $n\ge 1$ and $i=1, 2, \ldots, n$. 


\begin{theorem}\label{thm:main}  
Suppose that Condition (C) holds. If either of the following conditions is satisfied, then $\mu_{n} \overset{\mathbb{P}}{\to} \mu$ as $n$ tends to infinity.
\begin{enumerate}
  \item[(i)] For some $\mu_0 \in M$, it holds that $\sum_{i=1}^{n} \mathbb{P}(d(X_i, \mu_0) > n) \underset{n \to \infty}{\to} 0$ and $n^{-2}\sum_{i=1}^{n}\mathbb{E}[d^2\{\overline{X_{n,i}}(\mu_0), \mu_0\}] \underset{n \to \infty}{\to} 0$;

  \item[(ii)] The $X_i$'s are identically distributed and for some $\mu_0 \in M$, $n\mathbb{P}(d(X_1, \mu_0) > n) \underset{n \to \infty}{\to} 0$.
\end{enumerate}

\end{theorem}

Theorem~\ref{thm:main}  generalizes the sufficient direction of the classical result  (\ref{KF_wlln}) to random variables on non-compact symmetric spaces. 
However, their forms appear somewhat different. This discrepancy arises because while Theorem~\ref{thm:main} is stated under a symmetry assumption on $\mathbb{P}_{X_i}$, (\ref{KF_wlln}) remains valid even when $\mathbb{P}_{X}$ is non-symmetric. Notably, under the condition of (\ref{KF_wlln}), there are examples where $\mathbb{P}_{X}$ is not symmetric and $\mathbb{E}[X_1 1_{|X_1|\le n}]$ diverges as $n\to \infty$. Therefore, the classical case (\ref{KF_wlln}) includes the case where the sample mean $\bar{X}_n$ diverges. 
In contrast, if $\mathbb{P}_{X}$ defined on $\mathbb{R}$ is symmetric about $\mu$ as we have assumed in Theorem~\ref{thm:main}, then $\mathbb{E}[X_1 1_{|X_1|\le n}] \underset{n\to \infty}{\to} \mu$.

While our proof of Theorem \ref{thm:main} is provided in Appendix B, we point out three key steps in the proof. First, we observe and use the fact that if the tail conditions such as $\lim_{n\to\infty} n \mathbb{P}(d(X_1,\mu_0)>n)=0$ are satisfied for some $\mu_0\in M$, then they are satisfied for the center of symmetry $\mu$. Second, in the evaluation of $\mathbb{P}(d(\mu_n,\mu) > \epsilon)$, for a given $\epsilon>0$, the sample \Frechet\ mean $\mu_n$ is replaced by the sample \Frechet\ mean $\bar\mu_n = \mu_n\big(\overline{X_{n, 1}}(\mu), \overline{X_{n, 2}}(\mu), \ldots, \overline{X_{n, n}}(\mu)\big)$ of truncated random variables. Finally, we establish and utilize the fact that on a Hadamard space, the sample \Frechet\ mean $\bar\mu_n$ tends to shrink towards the center $\mu$ of symmetry, compared to a Euclidean counterpart. In fact, we establish this contraction argument for any deterministic data set $\{x_1, x_2, \ldots,x_n\} \subset M$ and for \textit{any base point} $x \in M$, which may not be the same as $\mu$. 


\begin{lemma}\label{lem:Hada}
Suppose that $(M,d)$ is a Hadamard manifold. For any $x,x_1, x_2,\ldots,x_n \in M$, 
\begin{equation}\label{ineq:key}
\|\mbox{Log}_{x}(\mu_{n}) \|_{x} \le  \|\frac{1}{n}\sum_{i=1}^{n}\mbox{Log}_{x}(x_i) \|_{x},
\end{equation}
where $\mu_{n}$ denotes the sample \Frechet\ mean of $(x_1, x_2, \ldots x_n) \in M^n$, and $\|\cdot\|_{x}$ denotes the Riemannian norm induced by the Riemannian metric $\langle ~,~\rangle_{x}$ on $T_{x}M$. The inequality in (\ref{ineq:key}) is strict unless $x,x_1,…,x_n$ all lie in a single totally geodesic flat submanifold of $M$.
\end{lemma}

An important question related to Theorem~\ref{thm:main} is whether the
converse of case~(ii) holds. This concerns the extent to which the Kolmogorov--Feller tail condition is tight for the weak law. Under Condition~(C), the converse holds in the flat case, that is, for non-compact symmetric spaces with zero curvature; see Appendix D. In general, however, we conjecture that the converse need not hold. A converse would follow if the inequality (\ref{ineq:key}) could be reversed, up to a multiplicative constant. Such a reversal may not be available in negatively curved spaces, where curvature effects may substantially alter the behavior of \Frechet\ means. We therefore leave the problem of proving a general nonlinear converse theorem, or constructing a counterexample, for future work.
Another potential direction for future work is to investigate Kolmogorov--Feller weak laws of large numbers for dependent random variables in the present non-compact symmetric setting, in
analogy with \cite{naderi2020weak}. 


\section{Examples}

We present examples where our result applies to i.i.d. sequences of random variables. 

\begin{example}[Symmetric positive-definite matrices]
Let \(M=\mathrm{Sym}^+(k)\), the space of $k \times k$ symmetric positive-definite matrices, be equipped with the affine-invariant Riemannian distance
\[
d_{\mbox{\tiny AI}}(A,B)=\|\log(A^{-1/2}BA^{-1/2})\|_{F}, \qquad (A,B \in M)
\]
where $\|\cdot\|_{F}$ stands for the Frobenius norm. Then $(M, d_{\mbox{\tiny AI}})$ is a non-compact symmetric space with non-positive curvature. At \(\mu=I_{k}\), the $k \times k$ identity matrix, the geodesic symmetry is given by $s_{\mu}(A)=A^{-1}$. Hence, a $k\times k$ positive-definite random matrix $X$ satisfying $X\stackrel{d}{=}X^{-1}$ is geodesically symmetric about $\mu$. In this setting, let $X, X_1, X_2, \ldots$ be a sequence of i.i.d. random variables. Theorem~\ref{thm:main}(ii) gives that for any $\mu_0 \in M$, $n\mathbb{P}(d_{\mbox{\tiny AI}}(X, \mu_0) > n) \underset{n \to \infty}{\to} 0$ implies $\mu_n(X_1, \ldots, X_n) \overset{\mathbb{P}}{\to} \mu$ as $n\to\infty$. 
\end{example}

\begin{example}[Product symmetric spaces]
Let $(M_j, d_j)_{j=1}^{K}$ be non-compact symmetric spaces and equip $M=M_1\times\cdots\times M_K$ with the product distance $$d(x,y)=\{\sum_{j=1}^K d_j^2(x_j,y_j)\}^{1/2}$$ for any $x=(x_1, \ldots, x_K)$ and $y=(y_1, \ldots, y_K)$. Then, $(M,d)$ is again a non-compact symmetric space. For a $\mu=(\mu_1,\ldots,\mu_K) \in M$, the geodesic symmetry about \(\mu\) is given by
\[
s_\mu(x_1,\ldots,x_K)
=
(s_{\mu_1}(x_1),\ldots,s_{\mu_{K}}(x_K)),
\]
where $s_{\mu_j}$ denotes the geodesic symmetry about $\mu_j \in M_j$. Let $\pi_j:M \to M_j, ~ 1\le j\le K$, be the natural projection map. Theorem~\ref{thm:main}(ii) applies to product-valued data whenever $X\stackrel{d}{=}s_\mu(X)$ and
$n\mathbb P(d(X,\mu)>n)\underset{n\to \infty}{\to} 0.$ The product-distance tail condition is equivalent to the coordinate-wise tail conditions $n\mathbb P\{d_{j}(\pi_{j}(X),\mu_j)>n\} \underset{n\to \infty}{\to} 0, \quad j=1, \ldots, K.$ Marginal geodesic symmetry alone does not imply joint geodesic symmetry in general, since the dependence structure may break the symmetry. However, if $\pi_1(X), \ldots, \pi_K(X)$ are independent, then $X\stackrel{d}{=}s_\mu(X)$ is equivalent to $\pi_j(X)\stackrel{d}{=}s_{\mu_{j}}(\pi_j(X))$ for all $1\le j\le K$. Thus, under independence, the geodesic symmetry condition can be verified coordinate-wise.  
\end{example}

We next present an illustrative example where our result applies to a non-i.i.d. sequence of random variables $(X_i)_{i=1}^{\infty}$ in $\mathrm{Sym}^+(k)$.  
In this example, although each $X_i$ has finite \Frechet\ variance $\sigma_i^2 := \inf_{p\in M} \mathbb{E}[d^2(X_i, p)]$, the sequence of variances $(\sigma_i^2)_{i=1}^\infty$ is unbounded. 


\begin{example} 
Consider a sequence $(X_i)_{i=1}^{\infty}$ of independent random variables taking values in $(\mathrm{Sym}^+(k),d_{\text{\tiny AI}})$ for some $k \ge 1$. Let $\mu = I_k$, the $k \times k$ identity matrix, be the center of symmetry. For a fixed $\alpha \in (0,1)$, each $X_i$ follows a log-Gaussian distribution:
$$\textsf{vec}\{\mbox{Log}_{\mu}(X_i)\} \sim N_{k(k+1)/2}(\textbf{0}, ~i^{\alpha}I_{k(k+1)/2}),$$ 
where $\textsf{vec}(\cdot)$ denotes the vectorization operator appropriate for symmetric $k \times k$ matrices \citep[][Section 3.5]{pennec2006riemannian}.  
The log-Gaussian distributions are geodesically symmetric about $\mu$, which is also the  \Frechet\ mean of each $X_i$, while the \Frechet\ variance $\sigma_i^2 = i^\alpha k(k+1)/2$ diverges as $i$ increases. 
In this setting, the conditions in Theorem~\ref{thm:main}(i) are satisfied (see Appendix C for details), and we conclude that $\mu_n(X_1,\ldots, X_n) \to \mu$ in probability as $n \to \infty$. 
\end{example}



For the case where only the first moment of $\mathbb{P}_{X}$ exists, the sample \Frechet\ mean almost surely converges to the minimizer of the \Frechet\ difference $F_{\mu}(m) = \mathbb{E}[d^2(X,m) - d^2(X,\mu)]$, which is well-defined for any $m \in M$ provided that $\mathbb{E}[d(X,\mu)]<\infty$ \citep{sturm2003probability,schotz2022strong}. 
We emphasize that Theorem~\ref{thm:main} relaxes this finite first moment condition. 
In the absence of moment assumptions, a symmetry condition serves to identify a center of the distribution. 
%
The following example demonstrates that
the tail condition in Theorem~\ref{thm:main} is genuinely weaker than the finite
first-moment condition.

\begin{example}
Suppose that $\mathbb{P}_{X}$ is a geodesically symmetric distribution about $\mu$ satisfying $$\mathbb{P}(d(X, \mu)>t) \propto \frac{1}{t\log(t)}$$ for large enough $t ~(\ge N)$. Then, it holds that $\lim_{n\to \infty}n\mathbb{P}(d(X, \mu)>n) = \lim_{n \to \infty}\frac{1}{\log(n)}= 0$. On the other hand, it is well-known that $\mathbb{E}[d(X, \mu)]=\infty$ is equivalent to $\sum_{n=N}^{\infty} \mathbb{P}(d(X, \mu) > n)= \sum_{n=N}^{\infty}\frac{c}{n\log(n)} =\infty$ for a fixed $c>0$, which follows from the integral test:
\begin{eqnarray*}
\sum_{n=N}^{\infty}\frac{1}{n\log(n)}=\infty ~ \Leftrightarrow ~ \int_{N}^{\infty}\frac{1}{x\log(x)}dx=\big[\log\{\log(x)\}\big]_{x=N}^{\infty} = \infty.
\end{eqnarray*}
Thus, by Theorem~\ref{thm:main}(ii), while the first moment does not exist for $\mathbb{P}_X$, its sample \Frechet\ mean $\mu_n$ converges to the center $\mu$ of symmetry in probability. 
\end{example}

\section*{Acknowledgement}  
We are grateful to the reviewers for their thoughtful and constructive comments. Their feedback led to improvements in the exposition, examples, and positioning of the results. Jongmin Lee was supported by the National Research Foundation of Korea (NRF) grants funded by the Korea government (MSIT) (RS-2026-25479875). Sungkyu Jung was supported by the National Research Foundation of Korea (NRF) grants funded by the Korea government (MSIT) (RS-2023-00301976 and RS-2024-00333399).




\begin{appendix}

\section{Proof of Proposition 1}
\begin{proof}[Proof of Proposition 1]
Suppose, for contradiction, that $\mu_1 \neq \mu_2$. Let $T = s_{\mu_2} \circ s_{\mu_1}$, and $\gamma$ be the geodesic from $\mu_1$ to $\mu_2$. Denote $\ell = 2d(\mu_1, \mu_2) > 0$. Then, $T$ is a translation-type isometry along $\gamma$, more precisely, a transvection. (For an illustration of the transvection, see Figure~\ref{fig:transvection}.) An application of Corollary 2.5 of \cite{sturm2003probability} implies that, since $M$ is Hadamard, we have for any $x \in M$, $d(x, s_{\mu_2}\circ s_{\mu_1}(x))\ge 2d(\mu_1, \mu_2)=\ell$.
%
%
This implies that, for any $x\in M$ and for any $m\in \mathbb{N}$, $d(x, T^m(x))\ge \ell m$. Thus, the isometry $T$, when repeatedly applied, pushes any point to infinity. 
\begin{figure}[htbp!]
\centering
\includegraphics[scale=0.45]{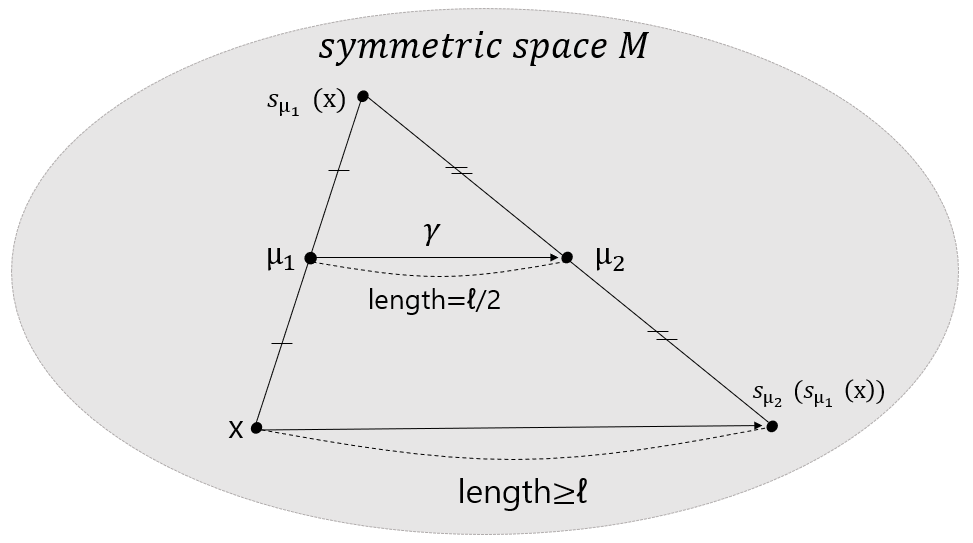}
\caption{An illustration of the transvection $T$, mapping from $x$ to $s_{\mu_2}\circ s_{\mu_1}(x)$.}
\label{fig:transvection}
\end{figure}

Given bounded sets $K$ and $A$, define a constant $C_{K, A}=\sup_{k\in K, a\in A}d(k, a) < \infty$. There exists an integer $m$ such that $\ell m> C_{K,A}$. For any $k \in K$ and for any $a \in A$, by the triangle inequality, 
$$
d(k, T^{m}(a)) \ge d(a, T^{m}(a)) - d(k,a) \ge \ell m - C_{K,A} >0,
$$
which implies that $K \cap T^{m}(A) =\emptyset$, where $T^{m}(A):=\{T^m(a): a \in A\}$. That is, for any bounded sets $K$ and $A$, there exists an integer $m_0$ such that for any integer $m'\ge m_0$, $K \cap T^{m'}(A) = \emptyset$, where for any $m\ge 1$, $T^{m}(A) \subset M$.

We now set the bounded sets $A$ and $K$ as follows. 
First, $A$ is chosen to be any bounded Borel set in $M$. Since $\mathbb{P}_X$ is geodesically symmetric about both $\mu_1$ and $\mu_2$, it holds that $\mathbb{P}_X(A) = \mathbb{P}_X(T(A))$.

Since $(M,d)$ is second countable, every probability measure on $M$ is tight \citep{lee2013smooth}. Hence, for any $\epsilon > 0$, there exists a compact set $K = K(\epsilon)$ satisfying $\mathbb{P}_X(K) \ge 1-\epsilon$. For all large $m$, we have $T^m(A) \cap K = \emptyset$, so that $\mathbb{P}(A) = \mathbb{P}(T^m(A)) \le 1 - \mathbb{P}(K) \le \epsilon$. 
Since $\epsilon$ is arbitrary, we have $\mathbb{P}(A) = 0$ for any bounded Borel set $A \subset M$. This is a contradiction to the tightness of $\mathbb{P}_{X}$, and therefore, $\mu_1 = \mu_2$.
\end{proof}

\section{Proofs of Theorem 1 and Lemma 1, and auxiliary results}
Our proof of Theorem 1 relies on several lemmas. We first state a result on the tail conditions imposed in Theorem 1. 

\begin{lemma}\label{lem:somevsall_tail_conditions} 
Let $(M,d)$ be a metric space, and let $(X_i)_{i=1}^\infty$ be a sequence of independent $M$-valued random variables. For an $m \in M$, let $c_1(m)$, $c_2(m)$ and $c_3(m)$ denote the following conditions, respectively: 
\begin{align*}
   c_1(m): & \lim_{n\to\infty} n\mathbb{P}(d(X_1,m) > n) = 0;\\
   c_2(m): & \lim_{n\to\infty} n^{-2}\sum_{i=1}^n\mathbb{E}[d^2\{\overline{X_{n,i}}(m),m\}] = 0;\\
   c_3(m): & \lim_{n\to\infty} \sum_{i=1}^n\mathbb{P}(d(X_i,m) > n) = 0.
\end{align*}
\begin{itemize}
    \item[(i)] For any $m,m' \in M$, $c_1(m)$ holds if and only if $c_1(m')$ holds.
    \item[(ii)] For any $m,m' \in M$, $c_2(m)$ holds if and only if $c_2(m')$ holds.
    \item[(iii)] Additionally assume that $M$ is a symmetric space and,  for a given $\mu \in M$, $\mathbb{P}_{X_i}$ is geodesically symmetric about $\mu$. Then, for any $m \in M$, $c_3(\mu)$ holds if $c_3(m)$ holds.
\end{itemize}
\end{lemma}

\begin{proof}[Proof of Lemma \ref{lem:somevsall_tail_conditions}]
(i) Let $m, m'$ be fixed, and without losing generality, assume that $R:= d(m,m') > 0$. We shall show that the condition $c_1(m)$ implies $c_1(m')$. The converse is obtained by the symmetry of the argument. 

Let $x \in M$ be any point. 
By the triangle inequality, 
\begin{equation}
   d(x,m') \le d(x, m) + d(m,m') = d(x, m) + R. \label{eq:triangle}
\end{equation} 
Therefore, for any $n > R$, $d(x,m') > n$ implies $d(x,m)> n - R$, which in turn implies that 
$$\{x \in M: d(x,m') > n\} \subset \{x \in M: d(x,m)> n - R\}.$$
Then, 
\begin{align*}
    n\mathbb{P}(d(X_1,m')>n) & \le n\mathbb{P}(d(X_1,m)>n-R) \\
      & = \frac{n}{n-R} (n-R) \mathbb{P}(d(X_1,m)>n-R).
\end{align*}
But, since we assumed $c_1(m)$, we have $(n-R) \mathbb{P}(d(X_1,m)>n-R) \to 0$ as $n\to\infty$, and $\lim_{n\to\infty}{n}/({n-R})= 1$ as $R$ is fixed. Thus, $n\mathbb{P}(d(X_1,m')>n) \to 0$ as $n\to\infty$, which is precisely the condition $c_1(m')$. 

(ii) Let $m, m'$ be fixed, and let $R:= \lceil d(m,m') \rceil > 0$ be the smallest non-zero integer larger than or equal to $d(m,m')$. As in Part (i), it suffices to show that $c_2(m)$ implies $c_2(m')$. 
Observe first that, for any $X_i,m \in M$, 
$$d^2(\overline{X_{n,i}}(m),m) = d^2(X_i,m) 1_{d(X_i,m) \le n},$$
which can be deduced from the definition of $\overline{X_{n,i}}(m)$. 

By the triangle inequality 
we have 
$$\{x \in M: d(x,m') \le n\} \subset \{x \in M: d(x,m)\le n + R\},$$ 
and, for any $x \in M$ 
$$d^2(x,m') \le \{d(x,m) + R\}^2 \le 2d^2(x,m) + 2R^2.$$
It follows that
\begin{align*}
    d^2(X_i,m') 1_{d(X_i,m') \le n} 
    & \le [2d^2(X_i,m) + 2R^2]1_{d(X_i,m') \le n} \\ 
    & \le [2d^2(X_i,m) + 2R^2]1_{d(X_i,m)\le n + R}.
\end{align*}

Write $F_m(n) : = n^{-2}\sum_{i=1}^n\mathbb{E}[d^2\{\overline{X_{n,i}}(m),m\}]$ for any $m\in M$ and $n\ge 1$. Then, 
\begin{align*}
    F_{m'}(n) & = n^{-2}\sum_{i=1}^n\mathbb{E} d^2(X_i,m') 1_{d(X_i,m') \le n}  \\
    & \le \frac{2}{n^2}\sum_{i=1}^n\mathbb{E} d^2(X_i,m)1_{d(X_i,m)\le n + R} + \frac{2R^2}{n^2}\sum_{i=1}^n \mathbb{P}(d(X_i,m)\le n + R) \\ 
    & \le \frac{2(n+R)^2}{n^2}F_{m}(n+R) + \frac{2R^2}{n^2}\sum_{i=1}^n 1.
\end{align*}
Since $F_{m}(n+R) \to 0$ as $n\to \infty$ (which is precisely the condition $c_2(m)$), $F_{m'}(n) \underset{n\to \infty}{\to} 0$ as desired.

(iii) Recall that $s_\mu(\cdot)$ is the geodesic symmetry with respect to $\mu \in M$. For any $x \in M$, we have 
$d(x, s_\mu(x)) = 2d(x,\mu) = 2d(s_\mu(x),\mu)$ (see Figure~\ref{fig:transvection} for an illustration). By the triangle inequality, for any $m\in M$, we obtain 
\begin{equation}
    \label{eq:iii} 
    2d(x,\mu) = d(x, s_\mu(x)) \le d(x,m ) + d(m, s_\mu(x)).
\end{equation}

Let us now fix $m$, and suppose that $c_3(m)$ holds. 
From (\ref{eq:iii}), we have 
\begin{align*}
\{x\in M: d(x,\mu) > n\} &\subset \{x\in M: d(x,m) + d(s_\mu(x), m) > 2n\}\\
&\subset \{x\in M: d(x,m) > n\}\cup \{x\in M: d(s_\mu(x),m) > n\}.
\end{align*}
Since $\mathbb{P}_{X_i}$ is geodesically symmetric about $\mu$, $X_i$ and $s_\mu (X_i)$ have the same distribution. Therefore, we have 
$$
\mathbb{P}(d(X_i,\mu) > n ) \le 2\mathbb{P}(d(X_i,m) > n).$$
Summing over $1\le i\le n$, we have 
$$
\sum_{i=1}^n\mathbb{P}(d(X_i,\mu) > n ) \le 2\sum_{i=1}^n\mathbb{P}(d(X_i,m) > n) \to 0,
$$
as $n\to \infty$. That is, $c_3(\mu)$ holds.
\end{proof}

The first two results of Lemma~\ref{lem:somevsall_tail_conditions} are valid for any metric space, and no symmetry condition is required. In contrast, the geodesic symmetry condition is required for Lemma~\ref{lem:somevsall_tail_conditions}(iii), and the converse is generally not true. That is, while $c_3(\mu)$ is true whenever $c_3(m)$ is true for some $m$, $c_3(m')$ may not be true even if $c_3(\mu)$ is true. A simple counterexample is the following: let $(M,d)$ be the usual real line with standard metric, and let $X_i = i$ with probability $1/2$ and $-i$ with probability $1/2$. The center of symmetry for $\mathbb{P}_{X_i}$ is $\mu = 0$. It can be checked that $\sum_{i=1}^n \mathbb{P}(d(X_i,\mu) > n) = \sum_{i=1}^n \mathbb{P}(|X_i - \mu| > n)= 0$ for all $n\ge 1$, but for $m' = 1$, $\sum_{i=1}^n \mathbb{P}(d(X_i,m') > n) = 1/2$ for all $n\ge 1$. 

Thanks to Lemma \ref{lem:somevsall_tail_conditions}, we shall show Theorem 1 under the following simplified condition:
\begin{enumerate}
  \item[(i)]  Condition (C) holds, $\sum_{i=1}^{n} \mathbb{P}(d(X_i, \mu) > n) \underset{n \to \infty}{\to} 0$ and \linebreak $n^{-2}\sum_{i=1}^{n}\mathbb{E}[d^2(\overline{X_{n,i}}, \mu)] \underset{n \to \infty}{\to} 0$;

  \item[(ii)] Condition (C) holds, the $X_i$'s are identically distributed, and it holds that $n\mathbb{P}(d(X_1, \mu) > n) \underset{n \to \infty}{\to} 0$.
\end{enumerate}

We next give a proof of Lemma 1. 

\begin{proof}[Proof of Lemma 1]
If $x=\mu_{n}$, equation (2) in the main paper is true since its left-hand side equals zero. Hence, we suppose that $x\neq \mu_n$. For a fixed $y \in M$ and for any geodesic $\gamma$ with same speed, define $f(t)=d^2(\gamma(t), y)$. Then, it is known that $f''(t) \ge 2\|\dot{\gamma}(t)\|^2_{\gamma(t)}$ when $M$ is a Hadamard manifold (see, e.g., \cite{do1992riemannian}). This inequality becomes strict when $M$ has negative curvatures everywhere and $y$ does not lie in $\gamma$. To employ the inequality, for each $i=1, 2, \ldots, n$, define $g_{i}(t) = d^2(\gamma(t), x_i)$ where $\gamma$ denotes the geodesic joining $x ~ (=\gamma(0))$ and $\mu_n ~(=\gamma(1))$ with same speed. Since $\gamma$ has the same speed, $\|\dot{\gamma}(t)\|_{\gamma(t)}=\|\mbox{Log}_{x}\mu_{n}\|_{x}$ for any $0\le t \le 1$. An application of mean-value theorem implies that for each $1\le i\le n$, $g'_i(1)-g'_i(0) \ge 2\|\mbox{Log}_{x}\mu_{n}\|^2_{x}$. The summation of the previous inequality from $i=1$ to $n$ gives
\begin{equation}\label{ineq:hadaineq}
\sum_{i=1}^{n} \{g'_{i}(1) - g'_{i}(0)\} \ge 2n\|\mbox{Log}_{x}(\mu_n) \|^2_{x}.
\end{equation}
For each $1\le i\le n$, differentiating each function $g_i(t)$, we obtain
\begin{align*}
g'_i(0)&= -2\langle\mbox{Log}_{x}(x_i),~ \dot{\gamma}(0) \rangle_{x}= -2\langle\mbox{Log}_{x}(x_i),~ \mbox{Log}_{x}(\mu_n)\rangle_{x}~ \mbox{and} \\
g'_{i}(1) &= -2\langle\mbox{Log}_{\mu_n}(x_i),~ \dot{\gamma}(1)\rangle_{\mu_n} = 2\langle\mbox{Log}_{\mu_n}(x_i),~ \mbox{Log}_{\mu_n}(x)\rangle_{\mu_n}. 
\end{align*}
In particular, we get $\sum_{i=1}^{n}g'_{i}(1)=\langle 2\sum_{i=1}^{n}\mbox{Log}_{\mu_n}(x_i),~ \mbox{Log}_{\mu_n}(x)\rangle_{\mu_n} = 0$, since $$\textsf{grad}|_{z=\mu_n}\{\sum_{i=1}^{n} d^2(z, x_i) \} = -2 \sum_{i=1}^{n} \mbox{Log}_{\mu_n}(x_i)=\textbf{0}.$$ With this in mind, according to (\ref{ineq:hadaineq}),
\begin{equation}\label{eq:log_inner}
-\sum_{i=1}^{n} g'_i(0)=2 \langle\sum_{i=1}^{n}\mbox{Log}_{x}(x_i),~ \mbox{Log}_{x}(\mu_{n}) \rangle_{x} \ge 2n\|\mbox{Log}_{x}(\mu_n) \|^2_{x}.
\end{equation}
Applying the Cauchy-Schwarz inequality, we obtain
\begin{equation}\label{eq:cs_ineq}
\|\sum_{i=1}^{n}\mbox{Log}_{x}(x_i) \|_{x} \cdot \|\mbox{Log}_{x}(\mu_n) \|_{x} \ge \langle \sum_{i=1}^{n}\mbox{Log}_{x}(x_i),~ \mbox{Log}_{x}(\mu_{n}) \rangle_{x}.
\end{equation}
Combining (\ref{eq:log_inner}) and (\ref{eq:cs_ineq}), we complete the proof. 

Note that equality in (\ref{eq:cs_ineq}) holds when $x,x_k,\mu_n$ are positioned on the same geodesic. In this regard, the inequality in equation (2) becomes strict when $M$ is a manifold with negative curvature and $x, x_{1}, x_{2}, \ldots, x_{n}$ are not located on a common geodesic.
\end{proof}

The following lemma serves as a step toward establishing the main results under Condition (C). 
\begin{lemma}\label{lem:modulation_ineq} 
Suppose that $(M,d)$ is a non-compact symmetric space and Condition (C) holds.
\begin{enumerate}
    \item[(a)] If, for each $1\le i\le n$, $\mathbb{P}_{X_{i}}$ has finite variance, i.e., $\mathbb{E}[d^2(X_{i}, \mu)] < \infty$, then 
     \begin{equation}\label{ineq:mod}
        \mathbb{E}[d^2(\mu_n, \mu)]\le \frac{1}{n^2}\sum_{i=1}^{n}\mathbb{E}[d^2(X_i, \mu)].
     \end{equation}

    \item[(b)] In particular, if $X_i$'s are identically distributed and $\mathbb{P}_{X}$ has finite variance, then
    \begin{equation}\label{ineq:modulation}
        \mathbb{E}[d^2(\mu_n, \mu)]\le \frac{1}{n}\mathbb{E}[d^2(X_1, \mu)].
    \end{equation}
    \end{enumerate}
\end{lemma}

\begin{proof}[Proof of Lemma~\ref{lem:modulation_ineq}]
(a) Since $(M,d)$ is a non-compact symmetric space, it is simply connected (see, e.g., \cite{helgason1979differential}). Thus, $M$ is a Hadamard manifold. Hence, we can employ the result of Lemma 1, which is presented in the main paper. Notably, the population \Frechet\ mean with respect to $\mathbb{P}_{X}$ is unique and equals $\mu$ since $M$ is a non-compact symmetric space (for details, see Proposition 2(a) in \cite{lee2026huber}). Thus, for each $i=1, 2, \ldots, n$, it holds that $\mathbb{E}[\mbox{Log}_{\mu}(X_i)] = \textbf{0} \in T_{\mu}M ~(\cong\mathbb{R}^{k})$ (since each $\mathbb{P}_{X_i}$ has finite variance). Substituting $x=\mu$, and $x_i=X_i$ for any $i=1,2,\ldots n$ into equation (2), we get
$$
\|\mbox{Log}_{\mu}(\mu_n) \|_{\mu} \le \|\frac{1}{n}\sum_{i=1}^{n} \mbox{Log}_{\mu}(X_i) \|_{\mu}.
$$
Squaring both sides of the above inequality and taking the expectation, we obtain the desired result.

(b) Since $X_{1}, X_2, \ldots, X_n$ has same distribution $\mathbb{P}_{X}$ by assumption, for any $1\le i\le n$, it holds that $\mathbb{E}[d^2(\mu, X_1)]=\mathbb{E}[d^2(\mu, X_i)]$, yielding (\ref{ineq:modulation}).

Notably, the inequality in (\ref{ineq:modulation}) becomes strict when $(M,d)$ has negative curvatures everywhere and $\mathbb{P}_X$ is not supported entirely on any single geodesic.
\end{proof}

 \begin{remark} 
         Suppose that Condition (C) holds, $X_i$ are identically distributed, and that $\mathbb{P}_{X}$ has finite variance. Then, (\ref{ineq:modulation}) implies that the variance modulation of \Frechet\ mean (see, for example, \cite{pennec2019curvature} and eq. (1.5) in \cite{hundrieser2024finite} for details), denoted by $\textsf{m}_{n} \ge 0$,
         is not greater than 1; that is, 
 \[
 \textsf{m}_{n}^2:=\frac{n\mathbb{E}[d^2(\mu_n, \mu)]}{\mathbb{E}[d^2(X, \mu)]} \le 1
 \]
 where $\textsf{m}_n < 1$ holds when $M$ has negative curvatures everywhere and $\mathbb{P}_{X}$ is not supported entirely on any single geodesic. The strict inequality $\textsf{m}_n < 1$ means that the rate of convergence for \Frechet\ mean is faster than the Euclidean counterpart. This result is more general than the existing one \citep[][Theorem 9]{pennec2019curvature} (when $M$ has constant negative curvatures), as it applies to a broader class of manifolds $M$ and underlying distributions $\mathbb{P}_{X}$ on $M$.
 \label{rmk:modulation}
 \end{remark}

\begin{lemma}\label{lem:moment}
For a real-valued random variable $Y$ and a fixed $p > 0$, 
$$
\mathbb{E}|Y|^{p}=\int_{0}^{\infty} pt^{p-1}\mathbb{P}(|Y|>t)dt.
$$
\end{lemma}
\begin{proof}[Proof of Lemma~\ref{lem:moment}]
For any non-negative random variable $Z$, it holds that $\mathbb{E}Z=\int_{0}^{\infty} \mathbb{P}(Z>t)dt$. By plugging-in $Z=|Y|^{p-1}$ into the previous equality,
\begin{align*}
\mathbb{E}|Y|^{p} &= \int_{0}^{\infty}\mathbb{P}(|Y|^p>u)du = \int_{0}^{\infty}\mathbb{P}(|Y|^p>t^p)pt^{p-1}dt \\
&=\int_{0}^{\infty} pt^{p-1}\mathbb{P}(|Y|>t)dt,
\end{align*}
where the second equality holds due to the integration by substitution for $u=t^{p}$.
\end{proof}
We are now ready to prove Theorem 1.

\begin{proof}[Proof of Theorem 1]

\textbf{Case I.} We begin by showing that if (i) in the statement of Theorem 1 holds, then $\mu_n \overset{\mathbb{P}}{\to} \mu$ as $n \to \infty$. Denote the sample \Frechet\ mean of $(\overline{X_{n, i}})_{1\le i\le n}$ by $\overline{\mu_{n}}$. Note that, for each $1\le i\le n$, the distribution of $\overline{X_{n, i}}$ is geodesically symmetric about $\mu$ and has finite variance. For a given $\epsilon > 0$, an application of Lemma~\ref{lem:modulation_ineq} implies that
\begin{eqnarray*}
\mathbb{P}(d(\mu_n, \mu ) > \epsilon) &\le& \mathbb{P}(\mu_n \neq \overline{\mu_n}) + \mathbb{P}(\mu_n=\overline{\mu_n}, ~ d(\mu_{n}, \mu) > \epsilon) \\
&\le& \mathbb{P}(\overline{X_{n, i}} \neq X_i ~ \mbox{for some} ~ i) + \mathbb{P}(d(\overline{\mu_n}, \mu) > \epsilon) \\
&\le& \sum_{i=1}^{n}\mathbb{P}(d(X_i, \mu) > n) + \mathbb{P}(d(\overline{\mu_n}, \mu) > \epsilon). \\
&=& \sum_{i=1}^{n}\mathbb{P}(d(X_i, \mu) > n) + \mathbb{E}1_{d(\overline{\mu_n}, \mu) > \epsilon} \\
&\le& \sum_{i=1}^{n}\mathbb{P}(d(X_i, \mu) > n) + \mathbb{E}\big[\frac{d(\overline{\mu_{n}}, \mu)}{\epsilon}\big]^2 \\
&\overset{(\because ~ Lemma~\ref{lem:modulation_ineq}(a))}{\le}& \sum_{i=1}^{n}\mathbb{P}(d(X_i, \mu) > n) + \frac{1}{n^2\epsilon^2}\sum_{i=1}^{n}\mathbb{E}[d^2(\overline{X_{n, i}}, \mu)] \\
&\underset{n \to \infty}{\to}& 0,
\end{eqnarray*}
where the last holds due to the hypothesis of Theorem 1 for the case (i). Therefore, $\mu_n$ converges in $\mu$ in probability as $n\to \infty$, as asserted.

\textbf{Case II.} We next prove that if the hypothesis of Theorem 1 for the case (ii) holds, then $\mu_n \overset{\mathbb{P}}{\to} \mu$ as $n \to \infty$. By assumption, for each $n\ge 1$, $X_{1}, X_{2}, \ldots X_{n}$ have a common distribution, and $\overline{X_{n,1}}, \overline{X_{n,2}}, \ldots, \overline{X_{n,n}}$ are also identically distributed. Due to \textbf{Case I}, it is enough to show that $\sum_{i=1}^{n} \mathbb{P}(d(X_i, \mu) > n) \underset{n \to \infty}{\to} 0$ and $n^{-2}\sum_{i=1}^{n}\mathbb{E}[d^2(\overline{X_{n, i}}, ~ \mu)] \underset{n \to \infty}{\to} 0$ are satisfied. 

It directly follows that $\sum_{i=1}^{n} \mathbb{P}(d(X_i, \mu) > n) = n \mathbb{P}(d(X_1, \mu) > n)\underset{n \to \infty}{\to} 0$. An application of Lemma~\ref{lem:moment} for $p=2$ implies that for any real-valued random variable $Y$, $\mathbb{E}(Y^2)=\int_{0}^{\infty} 2t\,\mathbb{P}(|Y|>t)dt$. With this in mind, observe that
\begin{align}\label{eq:int}
n^{-2}\sum_{k=1}^{n}\mathbb{E}[d^2(\overline{X_{n,k}},\mu)] &= n^{-1}\mathbb{E}[d^2(\overline{X_{n,1}}, \mu)] \nonumber \\
&=n^{-1} \int_{0}^{\infty} 2t\, \mathbb{P}(d(\overline{X_{n,1}}, \mu) > t)dt \nonumber \\
&\le 2n^{-1}\int_{0}^{n}t\, \mathbb{P}(d(X_1, \mu) > t)dt 
\end{align}
The remainder of this proof is to verify that (\ref{eq:int}) goes to zero as $n \to \infty$. To do this, note that for any $\epsilon > 0$, there exists $T(\epsilon)>0$ such that $\mbox{sup}_{t\ge T(\epsilon)} t\mathbb{P}(d(X_1, \mu) > t) \le \epsilon$ (due to the assumption $n\mathbb{P}(d(X_1, \mu) > n) \underset{n \to \infty}{\to} 0$). Hence, for a fixed $\epsilon > 0$,
\begin{align*}
n^{-1}\int_{0}^{n} t\mathbb{P}(d(X_1, \mu) > t)dt =& \underbrace{n^{-1}\int_{0}^{T(\epsilon)} t\mathbb{P}(d(X_1, \mu) > t)dt}_{\le T(\epsilon)/n} \\ 
&+ \underbrace{n^{-1}\int_{T(\epsilon)}^{n} t\mathbb{P}(d(X_1, \mu) > t)dt}_{\le \epsilon}. 
\end{align*}
Taking $\limsup_{n\to \infty}$ on both sides above, we get 
\[
\limsup_{n\to\infty}n^{-1}\int_{0}^{n} t\mathbb{P}(d(X_1, \mu) > t)dt \le \epsilon.  
\]
Since $\epsilon$ is arbitrary, letting $\epsilon \to 0$ gives $\lim_{n\to\infty}n^{-1}\int_{0}^{n} t\mathbb{P}(d(X_1, \mu) > t)dt = 0$. Therefore, (\ref{eq:int}) goes to zero as $n\to\infty$, as desired.
\end{proof}

\section{Technical details for Example 3}

%
Recall that for each $i\ge 1$, $\textsf{vec}\{\mbox{Log}_{\mu}(X_i)\} \sim N_{k(k+1)/2}(\textbf{0}, ~i^{\alpha}I_{k(k+1)/2})$. 
The Chernoff bound implies that for any $t \in (0, 1/2)$ and for any $x>0$, $\mathbb{P}(\chi^2_{k(k+1)/2} > x) \le e^{-tx}(1-2t)^{-k(k+1)/4}$, where $\chi^2_{k(k+1)/2}$ stands for the chi-squared distribution with degrees of freedom $k(k+1)/2$. In particular, setting $t=1/4$ gives that for any $x > 0$, $\mathbb{P}(\chi^2_{k(k+1)/2} > x) \le 2^{k(k+1)/4}e^{-x/4}$. With this in mind, we get
\begin{align*}
\sum_{i=1}^{n}\mathbb{P}(d(X_i, \mu) > n)&= \sum_{i=1}^{n}\mathbb{P}(\|\mbox{Log}_{\mu}(X_i)\|_{\mu}^2 > n^2) \\
&= \sum_{i=1}^{n}\mathbb{P}(\|\textsf{vec}\{\mbox{Log}_{\mu}(X_i)\}\|_{2}^2 > n^2) \\
&= \sum_{i=1}^{n}\mathbb{P}(\chi^2_{k(k+1)/2}>n^2/i^\alpha) \\
&\le 2^{k(k+1)/4}\sum_{i=1}^{n}\exp(-\{n^2/(4i^\alpha)\}) \\
& < 2^{k(k+1)/4}n\,\exp(-n^{(2-\alpha)}/4) \\
& \underset{n \to \infty}{\to} 0,
\end{align*}
where $\|\cdot \|_{2}$ denotes the $L_2$-norm in $\mathbb{R}^{k(k+1)/2}$. In addition, 
\begin{align*}
n^{-2}\sum_{i=1}^{n}\mathbb{E}[d^2(\overline{X_{n,i}},\mu)] 
&\le n^{-2}\sum_{i=1}^{n}\mathbb{E}[d^2(X_i,\mu)] = n^{-2}\sum_{i=1}^{n}\mathbb{E}[\|\mbox{Log}_{\mu}(X_i)\|_{\mu}^2] \\
&=n^{-2}\sum_{i=1}^{n}\textsf{tr}[\textsf{Cov}\{\textsf{vec}(\mbox{Log}_{\mu}(X_i))\}] = \frac{k(k+1)}{2}n^{-2}\sum_{i=1}^{n}i^\alpha \\
&\le \frac{k(k+1)}{2} n^{-2}\int_{1}^{n+1}x^\alpha dx \\
&= \frac{k(k+1)}{2}n^{-2}\{(n+1)^{\alpha+1} -1\}/(\alpha+1) \\
&< \frac{k(k+1)}{2}n^{\alpha-1}(1+ 1/n)^{\alpha + 1}/(\alpha+1) \\
&\underset{n \to \infty}{\to} 0.
\end{align*}
Since the hypotheses of Theorem 1(i) hold, $\mu_n(X_1, X_2, \ldots , X_n) \overset{\mathbb{P}}{\to} \mu$ as $n\to \infty$ by Theorem 1(i).

\section{A partial converse for Theorem 1(ii) and its proof.}

An interesting question is whether the converse of Theorem 1 for case (ii) holds. This concerns the extent to which the Kolmogorov--Feller tail condition is tight for the weak law. In the Euclidean space the following result is classical; it follows from
an application of Corollary 3.3 in \cite{marcus1979stable}. 
For completeness, we include a proof in this section. 

\begin{theorem}\label{thm:main2}
 Suppose that $M=\mathbb{R}^k$ is endowed with an arbitrary inner product, and that Condition (C) holds and the $X_i$ are identically distributed. If $\mu_n \overset{\mathbb{P}}{\to} \mu$ as $n \to \infty$, then $n\mathbb{P}(\|X_1-\mu\| > n) \underset{n \to \infty}{\to} 0$.
\end{theorem}
Since a noncompact symmetric space with identically zero sectional curvature is isometric to a Euclidean space, the result established in the Euclidean setting applies verbatim to this flat symmetric-space case. 
We conjecture that the result of Theorem \ref{thm:main2} does not hold for general non-compact symmetric spaces.

To verify Theorem~\ref{thm:main2}, we will use the following lemma regarding a symmetric distribution on the real line. 
\begin{lemma}\label{lem:same_dist}
Given a real-valued random variable $Y$, suppose that $Y$ and $-Y$ have the same distribution; that is, $Y\overset{d}{=}-Y$. Then, $\mathbb{P}(Y\ge 0) \ge 1/2$. 
\end{lemma}

\begin{proof}[Proof of Lemma~\ref{lem:same_dist}]
Since $Y$ and $-Y$ have the same distribution, an application of countable sub-additivity of probability measures yields that 
\begin{align*}
1=\mathbb{P}(Y\in \mathbb{R}) &\le \mathbb{P}(Y \ge 0) + \mathbb{P}(Y\le 0) \\
&= \mathbb{P}(Y \ge 0) + \mathbb{P}(-Y\le 0) \\
&\le 2\mathbb{P}(Y\ge 0).
\end{align*}
In turn, this leads to the desired result.
\end{proof}

\begin{proof}[Proof of Theorem~\ref{thm:main2}]
When $k=1$, it is well-known that the statement of Theorem~\ref{thm:main2} holds; see, for example, \cite{stoica2010kolmogorov, yuan2015conditional, naderi2019version}. For this reason, suppose that $M$ is equal to $\mathbb{R}^k$, $k\ge 2$, endowed with an arbitrary inner product $\langle~, ~\rangle$ and the induced norm $\|\mathbf{x}\|:=\langle \mathbf{x}, \mathbf{x}\rangle^{1/2}$ for $\mathbf{x}\in \mathbb{R}^k$.

Without loss of generality, we may assume that $\mu=\textbf{0}\in \mathbb{R}^k$. (If not, for any $1\le i\le n$, replace $X_i$ with $X_i - \mu$.) Suppose, for contradiction, that there exists a $\delta > 0$ and an increasing subsequence of natural numbers $(n_{j})_{j\ge 1}$ such that for any $j\ge 1$, $n_j\mathbb{P}(\|X_1\| > n_j) \ge \delta$. Notably, $\mathbb{P}(\|X_1\|>n)$ is close to zero for large enough $n$ and an inequality $(1-p)^n \le e^{-np}$ holds for small $p\ge0$. With these in mind, for large enough $j$
\begin{align}\label{ineq:strict}
\mathbb{P}(\exists 1 \le i\le n_{j}: \|X_i\|> n_{j}) &= 1- \mathbb{P}(\max_{1\le i\le n_{j}} \|X_i\| \le n_{j}) \nonumber \\
&= 1-\{\mathbb{P}(\|X_1\| \le n_{j})\}^{n_{j}} \nonumber \\
&= [1-\{1-\mathbb{P}(\| X_1\|>n_{j})\}]^{n_{j}} \nonumber \\
&\ge 1-e^{-n_{j}\mathbb{P}(\|X_1 \|>n_{j})} \nonumber \\
&\ge 1-e^{-\delta} ~(>0).
\end{align}
Let $j$ be a fixed large natural number and denote $S_{n_{j}}=\sum_{\ell=1}^{n_{j}}X_\ell$ and $S_{-i} = S_{n_{j}} - X_i ~ (1\le i\le n_j)$. An application of (\ref{ineq:strict}) yields that
\begin{align}\label{ineq:strict2}
\mathbb{P}(\|S_{n_j}/n_{j}\|>1) \ge& \mathbb{P}(\exists1\le i\le n_j: \|X_i\|>n_j ~ \mbox{and} ~ \langle X_i, S_{-i} \rangle \ge 0) \nonumber \\
=& \mathbb{P}(\langle X_i, S_{-i}\rangle \ge 0  \mid \exists1\le i\le n_j: \|X_i\|>n_j) \nonumber \\
& \times \mathbb{P}(\exists1\le i\le n_j: \|X_i\|>n_j) \nonumber \\
\ge& 1/2 \times \mathbb{P}(\exists 1\le i\le n_j: \|X_i\|>n_j) \nonumber \\
\ge& (1-e^{-\delta})/2 ~( > 0),
\end{align}
where the first inequality holds due to the inequality below:
$$
\|S_{n_{j}}\|^2=\langle X_i + S_{-i}, X_i+S_{-i} \rangle =\|X_i\|^2 + 2\langle X_i, S_{-i} \rangle +\|S_{-i} \|^2 > n_j^2.
$$
Moreover, the second inequality of (\ref{ineq:strict2}) is satisfied due to an application of Lemma~\ref{lem:same_dist} to the fact that $\langle X_i, S_{-i}\rangle \overset{d}{=} \langle X_i, -S_{-i} \rangle$ (recall that $\mathbb{P}_{X}$ is geodesically symmetric about $\mathbf{0}$, thereby implying $S_{-i}\overset{d}{=}-S_{-i}$). Since $\mu_n=S_n/n$ and (\ref{ineq:strict2}) holds for any sufficiently large $j$, $\mu_n$ does not converge to $\mu ~(=\mathbf{0})$ in probability as $n\to \infty$. This contradicts the assumption.
\end{proof}

\section{Sharpness of Theorem 1(ii)} 

The Kolmogorov--Feller tail condition, namely $n\mathbb{P}(d(X_1, \mu_0) > n) \to 0$, implies that $\mathbb{E}[d^{\alpha}(X, \mu)] <\infty$ for all $\alpha \in (0, 1)$. Therefore, under the same conditions as in Theorem 1(ii), the $\alpha$-\Frechet\ mean 
$$\mu_n^{\alpha}:=\arg\min_{m \in M} n^{-1}\sum_{i=1}^{n}d^\alpha(X_i, m)$$ 
for any $\alpha \in (0,1)$, converges to $\mu$ almost surely as $n$ increases \cite[][Corollary 5.2]{schotz2022strong}. Nevertheless, Theorem 1(ii) is sharp in the sense that, under the same Kolmogorov--Feller tail condition, the convergence in probability established in Theorem 1 is not, in general, strengthened to almost sure convergence. The following example demonstrates this.

\begin{example}[Failure of almost sure convergence]
Let $(M, d)$ be a non-compact symmetric space, fix $\mu \in M$, and let $\gamma:\mathbb{R}\to M$ be a unit-speed geodesic with $\gamma(0)=\mu$. Let $Z$ be a symmetric real-valued random variable satisfying
$$
\mathbb{P}(Z = \pm 2^i) = \frac{1}{(2i)2^i} \quad (i \ge 1), \quad \mbox{and} \quad \mathbb{P}(Z=0) = 1 - \sum_{i=1}^\infty \frac{1}{i(2^i)}.
$$
Define the $M$-valued random variable $X= \gamma(Z)$, and let $X, X_1, X_2, \ldots$ be independent and identically distributed random variables. Then, $\mathbb{P}_X$ is geodesically symmetric about $\mu$. Since $\gamma$ is unit-speed, $d(X,\mu) = |Z|$. The Kolmogorov--Feller condition is satisfied:
$
n\mathbb{P}(d(X,\mu)>n) = n\mathbb{P}(|Z|>n) \to 0
$
as $n\to \infty$, which in turn implies that $\mu_n \to \mu$ in probability. On the other hand, since all observations lie on the geodesic $\gamma(\mathbb{R})$, the sample Fréchet mean $\mu_n=\mu_n(X_1, X_2, \ldots, X_n)$ coincides with the Euclidean sample mean along that geodesic, that is,
$$
\mu_n = \gamma(\bar Z_n),
\qquad
\bar Z_n := \frac{1}{n}\sum_{i=1}^n Z_i,
$$
where $Z, Z_1, Z_2, \ldots$ are independent and identically distributed random variables. However, $\bar Z_n \not\to 0$ almost surely. Thus,
$$
\mu_n \not\to \mu \qquad \text{almost surely}.
$$
Therefore, the conditions in Theorem 1(ii) do not, in general, guarantee almost sure convergence.
\end{example}
Technical details for Example E.1 are described below. For any integer $i\ge 1$,
\begin{align*}
2^i \mathbb{P}(|Z| > 2^i) &= \sum_{m=1}^{\infty} \frac{1}{i+m}\cdot\frac{1}{2^m} \\
&< \frac{1}{i+1} (\sum_{m=1}^{\infty} \frac{1}{2^m}) \\
&= \frac{1}{i+1} \\
&\underset{i \to \infty}{\to} 0.
\end{align*}
For each $n\ge 2$, there exists a unique integer $i\ge 1$ such that $2^i \le n <2^{i+1}$. Then,
$$
n\mathbb{P}(|Z| > n) <2^{i+1}\mathbb{P}(|Z|>2^i) = 2 \cdot 2^{i}\mathbb{P}(|Z|>2^i). 
$$
As $n$ increases, so does $i$. Combining these facts, we have
$$
n\mathbb{P}(d(X,\mu)>n) = n\mathbb{P}(|Z|>n) \underset{n\to \infty}{\to} 0.
$$
So the Kolmogorov--Feller tail condition is fulfilled. 
Also, recall that 
$$
\mu_n = \gamma(\bar Z_n),
\qquad
\bar Z_n := \frac{1}{n}\sum_{i=1}^n Z_i,
$$
where $Z, Z_1, Z_2, \ldots$ are independent and identically distributed random variables. 

Now, we wish to verify that $\sum_{n=1}^{\infty}\mathbb{P}(|Z_n| > n)=\infty$. By Cauchy's condensation test, 
$$
\sum_{n=1}^{\infty}\mathbb{P}(|Z_n| > n)= \sum_{n=1}^{\infty}\mathbb{P}(|Z| > n)=\infty \quad \Leftrightarrow \quad \sum_{n=1}^{\infty} 2^n \mathbb{P}(|Z|>2^n) = \infty.
$$
Since $2^n\mathbb{P}(|Z|>2^n)= \sum_{m=1}^{\infty} \frac{1}{2^m(n+m)} > \frac{1}{2(n+1)}$ and $\sum_{n=1}^{\infty}\frac{1}{2(n+1)}=\infty$, $$\sum_{n=1}^{\infty}\mathbb{P}(|Z_n| > n) =\infty.$$ 
By independence on $(Z_n)_{n\ge 1}$ and the second Borel--Cantelli lemma,
$$
\mathbb{P}(\underbrace{|Z_n|>n \quad \text{infinitely many} ~ n}_{=:E}) = 1,
$$ 
On the event $E$, there exists a subsequence $(n_k)_{k\ge 1}$ such that $\frac{|Z_{n_k}|}{n_k} > 1$, and hence $Z_n/n \not\to 0$. In particular, $\bar Z_n \not\to 0$ whenever $E$ occurs. Consequently,
$$
\mu_n \not\to \mu,
$$
almost surely.

\end{appendix}

\bibliographystyle{imsart-nameyear}
\bibliography{KFLLN}     

\end{document}